\documentclass[11pt,a4paper]{amsart}
\usepackage{hyperref}
\hypersetup{
	colorlinks,
	linkcolor=black,
	citecolor=black,
	urlcolor=black,
	breaklinks=true
}
\usepackage[utf8]{inputenc}
\usepackage[T1]{fontenc}
\usepackage[english]{babel}
\usepackage{
	amssymb,
	mathtools,
	todonotes
}
\usepackage{geometry}
\newtheoremstyle{lemma}{.5\baselineskip\@plus.2\baselineskip\@minus.2\baselineskip}{.5\baselineskip\@plus.2\baselineskip\@minus.2\baselineskip}
	{\itshape}
	{}
	{\bfseries}
	{.}
	{\newline}
	{\thmname{#1}\thmnumber{ #2}\thmnote{ (#3)}}
\theoremstyle{lemma}
	\newtheorem{theorem}{Theorem}
	\newtheorem{lemma}[theorem]{Lemma}
	\newtheorem{proposition}[theorem]{Proposition}
	\newtheorem{corollary}[theorem]{Corollary}

\newtheoremstyle{definition}{.5\baselineskip\@plus.2\baselineskip\@minus.2\baselineskip}{.5\baselineskip\@plus.2\baselineskip\@minus.2\baselineskip}
	{}
	{}
	{\bfseries}
	{.}
	{\newline}
	{\thmname{#1}\thmnumber{ #2}\thmnote{ (#3)}}
\theoremstyle{definition}

\newcommand{\M}{\ensuremath{\mathcal{M}}}

\newcommand{\Z}{\ensuremath{\mathbb{Z}}}

\newcommand{\R}{\ensuremath{\mathbb{R}}}

\newcommand{\sphere}{\ensuremath{\mathbb{S}}}
\newcommand{\HM}{\ensuremath{\mathcal{H}}}

\newcommand{\dd}{\ensuremath{\mathrm{d}}}

\newcommand{\abs}[1]{\left| #1 \right|}

\renewcommand{\phi}{\varphi}
\renewcommand{\epsilon}{\varepsilon}

\hypersetup{
	pdftitle={A Sharp Lower Bound for Integral Menger Curvature},
	pdfsubject={},
	pdfauthor={Alexander Dohmen, Sebastian Scholtes},
	pdfkeywords={integral Menger curvature, knot energy, singular integral, Menger curvature, round circle},
	pdfcreator={},
	pdfproducer={}
}

\makeindex
\begin{document}

\title{A Sharp Lower Bound for Integral Menger Curvature}
\author{Alexander Dohmen}
\address{Institut für Mathematik, RWTH Aachen University, Im Süsterfeld 2, Aachen 52072, Germany}
\email{\href{mailto:dohmen@instmath.rwth-aachen.de}{dohmen@instmath.rwth-aachen.de}}
\urladdr{\href{https://www.instmath.rwth-aachen.de/~dohmen/home/}{https://www.instmath.rwth-aachen.de/~dohmen/home/}}
\author{Sebastian Scholtes}
\address{Liberal Arts and Sciences, Technische Hochschule Augsburg, An der Hochschule 1, Augsburg 86161, Germany}
\email{\href{mailto:sebastian.scholtes@tha.de}{sebastian.scholtes@tha.de}}
\urladdr{\href{https://www.tha.de/homes/sescho/}{https://www.tha.de/homes/sescho/}}
\date{\today}
\keywords{integral Menger curvature, knot energy, Menger curvature, singular integral, knots, round circle}
\subjclass[2020]{49Q10; 53A04, 42B20, 57K10}
\begin{abstract}
	We prove that round circles $C_L$ are the unique minimisers
	of integral Menger curvature among embedded closed
	$C^1$ curves $\gamma$ of prescribed length $L$ for every $p\geq 2$, i.e.
	\begin{align*}
		\M_p(C_L) \leq \M_p(\gamma)
	\end{align*}
	with equality if and only if $\gamma$ parametrises
	a round circle.
	The proof combines a three-point symmetrisation,
	an analysis of diagonal truncations, and the
	periodic Wirtinger inequality.
	For $p>3$, known self-avoidance and regularity
	results extend the minimisation result to closed
	rectifiable curves without assuming embeddedness
	or smoothness, provided the length of the curve
	image equals $L$.
	
	In particular, $\M_p$ detects the unknot and is basic for $p>3$, thereby answering a question posed repeatedly in the literature.
\end{abstract}
\maketitle

\section{Introduction}

For three pairwise distinct points $x, y, z\in\R^d$,
let $R(x, y, z)$ denote the radius of their
circumcircle, with $R(x, y, z)=\infty$ if the points
are collinear. Their \emph{Menger curvature} is
defined by
\begin{align*}
    c(x, y, z)\vcentcolon=\frac{1}{R(x, y, z)}.
\end{align*}
Menger curvature is a classical three-point
quantity with applications in harmonic analysis, geometric measure
theory and geometric knot
theory.

In harmonic analysis, integrals of squared Menger curvature
are closely connected to Melnikov's symmetrisation
of the Cauchy kernel and the study of analytic
capacity
\cite{Melnikov1995,MelnikovVerdera1995,
MattilaMelnikovVerdera1996}.
These ideas played an important role in
Tolsa's solution of the Painlev\'e problem
\cite{Tolsa2003}.
A precise identity relating the curvature of a measure with linear growth
and the squared $L^2$ norm of its Cauchy transform is due to Melnikov and
Verdera \cite{MelnikovVerdera1995}. Tolsa and Verdera \cite{TolsaVerdera2006}
used Menger curvature to study whether the Cauchy transform of a non-trivial
finite measure can vanish on its support.

In geometric measure theory, integral Menger
curvature has been used to study rectifiability
and regularity of sets.
L\'eger \cite{Leger1999} proved that finite
quadratic integral Menger curvature implies
rectifiability of one-dimensional sets in
Euclidean space, a result extended to general
metric spaces by Hahlomaa \cite{Hahlomaa2008}.
Lin and Mattila \cite{LinMattila2001} studied
the regularity of fractal sets using integrals
of powers of Menger curvature.
Schul \cite{Schul2007} investigated the
connection between averaged Menger curvature,
quantitative flatness, and the analyst's
travelling salesman problem for Ahlfors-regular
curves in metric spaces. 
More recently, Azzam and Tolsa
\cite{AzzamTolsa2015} investigated the
relationship between Jones' square functions
and quadratic Menger curvature of measures
with linear growth.
Graczyk, Jones, and Mihalache
\cite{GraczykJonesMihalache2022} studied
measures with bounded Menger curvature
and their connections to Hausdorff dimension
and analytic capacity.
F\"assler and Violo \cite{FasslerViolo2025}
obtained further results on quantitative
rectifiability in metric spaces, including
a criterion involving local integrals of
squared Menger curvature.
At the interface with geometric function theory, Bishop \cite{Bishop2025} characterised Weil--Petersson curves in terms of several geometric quantities, including Jones' $\beta$-numbers and integrals involving squared Menger curvature.

In geometric knot theory, the global radius of curvature introduced by Gonzalez and Maddocks
\cite{GonzalezMaddocks1999} led to a family of nonlocal curvature energies that control both the geometry and the topology
of curves; see \cite{StrzeleckiSzumanskaVdM2010,
StrzeleckiSzumanskaVdM2013,StrzeleckiVdM2013,StrzeleckiVdM2018}.

Let $L>0$, write $\sphere_L=\R/L\Z$, and let
$\gamma:\sphere_L\to\R^d$ be a closed rectifiable
curve parametrised by arc length. For $p>0$, its
\emph{integral Menger curvature} is
\begin{align*}
    \M_p(\gamma)
    \vcentcolon=
    \iiint_{\sphere_L^3}
    c\bigl(\gamma(s),\gamma(t),\gamma(u)\bigr)^p
    \,\dd s\,\dd t\,\dd u,
\end{align*}
where the value $+\infty$ is allowed and the
integrand may be set to zero whenever two of
the three points coincide; this convention
does not affect the integral. Under a dilation by $\lambda>0$,
the energy scales according to
\begin{align*}
    \M_p(\lambda\gamma)
    =\lambda^{3-p}\M_p(\gamma).
\end{align*}
Thus $p=3$ is the scale-invariant exponent.
For a round circle $C_L$ of length $L$, every
triple of distinct points has circumradius
$\frac{L}{2\pi}$, and therefore
\begin{align}\label{circlevalue}
    \M_p(C_L)=(2\pi)^pL^{3-p}.
\end{align}

In the following, we focus on results concerning the classical integral Menger curvature $\M_p$, rather than its various generalizations. For broader discussions of related geometric curvature energies, we refer to \cite{StrzeleckiSzumanskaVdM2013,StrzeleckiVdM2013,StrzeleckiVdM2018}.

Integral Menger curvature has been studied
extensively in the supercritical range $p>3$.
Strzelecki, Szuma\'nska, and von der Mosel
\cite{StrzeleckiSzumanskaVdM2010}
established self-avoidance
and regularity results. 
They showed that finite $\M_p$-energy excludes
self-intersections for $p\geq 3$ (see \cite{StrzeleckiVdM2013} for $p = 3$), provided multiple coverings are excluded. 
For $p > 3$, they additionally established $C^{1, 1-3/p}$ regularity of arc length parametrisations.
Blatt \cite{Blatt2013} subsequently
characterised finite $\M_p$-energy among
embedded $C^1$ curves by membership in the
Sobolev--Slobodeckij space $W^{2-2/p, p}$.

The existence of energy minimisers in prescribed
knot classes was established in
\cite{StrzeleckiSzumanskaVdM2010,StrzeleckiSzumanskaVdM2013,StrzeleckiVdM2013}.
They also showed that, for suﬃciently large $p$, every global minimiser is unknotted \cite[Theorem 3.9]{StrzeleckiSzumanskaVdM2013}.
Hermes \cite{Hermes2012} showed that the round circle is a stationary point.
Blatt and Reiter subsequently established $C^1$-differentiability on suitable spaces of embedded curves for $p>3$ \cite[Theorem~3]{BlattReiter2015}. Their regularity result for stationary points, however, does not apply to the classical integral Menger curvature $\M_p$.
Knappmann \cite{Knappmann2020} proved local minimality of the round circle for $p>3$ by analysing the second variation.
Gilsbach \cite{Gilsbach2018} proved the existence of distinct critical points in torus knot classes and investigated their behaviour as $p\to\infty$ in relation to ropelength.
However, these results did not establish
global minimality. In fact, whether the round
circle is the unique absolute minimiser of
$\M_p$ among closed curves of prescribed length has repeatedly been posed as an open problem, even
for planar curves; see \cite[Section~1.1]{StrzeleckiSzumanskaVdM2013}, \cite[Problem~5.1]{StrzeleckiVdM2013}, \cite[p.~1353]{OWR2013} and \cite[Open Questions~2.2.13]{StrzeleckiVdM2018}.

Numerical and discrete approximations of integral
Menger curvature have also been investigated.
Hermes \cite{Hermes2012} developed a numerical
gradient flow based on trigonometric polynomial
approximations.
Building on this approach, Gilsbach
\cite{Gilsbach2018} investigated symmetric
critical points in torus knot classes numerically,
in particular their stability and their behaviour
as $p\to\infty$.
Scholtes \cite{Scholtes2014Thesis,Scholtes2018Discrete}
studied polygonal discretisations and established
corresponding $\Gamma$-convergence results.

The present paper establishes a sharp global estimate throughout the embedded $C^1$ class,
starting at the quadratic exponent.

\begin{theorem}[Sharp lower bound for integral Menger curvature]\label{mainTheorem}
    Let $d \geq 2$ and let $\gamma\in C^{1}(\sphere_{L}, \R^{d})$ be parametrised by arc length and embedded.
    Then for $p \geq 2$,
    \begin{align}\label{mainInequality}
        (2\pi)^pL^{3-p} \leq \M_p(\gamma).
    \end{align}
    Equality holds iff $\gamma$ parametrises a round circle.
\end{theorem}

In the supercritical range $p>3$, known
self-avoidance and regularity results allow
the conclusion to be extended to rectifiable
curves without imposing differentiability
or embeddedness a priori.
Since integral Menger curvature alone does
not exclude multiple coverings, we require
the length of the curve image to agree with
the total traversed length.
Combining the regularity and topological
results of
\cite{StrzeleckiSzumanskaVdM2010}
with Theorem~\ref{mainTheorem} gives the
following consequence.

\begin{corollary}[The round circle is the unique absolute minimiser of $\M_p$ for $p>3$]\label{basicCorollary}
	Let $d\geq 2$ and let $\gamma:\sphere_L\to\R^d$ be a closed rectifiable curve parametrised by arc length, such that
	\begin{align*}
		\HM^1\bigl(\gamma(\sphere_L)\bigr)=L.
	\end{align*}
	Then for all $p>3$ we have
	\begin{align*}
		\M_p(C_L) \leq \M_p(\gamma),
	\end{align*}
	with equality iff $\gamma$ parametrises a round circle.
	In particular, $\M_p$ is unknot-detecting and basic.
\end{corollary}

At the scale-invariant exponent $p = 3$, finite integral Menger curvature already excludes self-intersections for curves without multiple coverings \cite[Remark after Lemma 2.4]{StrzeleckiVdM2013}. However, the available
regularity results do not directly provide the $C^1$ regularity required by our truncation argument.
For results on weak approximate tangents at the critical exponent, see \cite[Theorem 1.1]{Scholtes2012Tangency}.

We now describe the argument leading to
Theorem~\ref{mainTheorem}.
The starting point is the quadratic energy
$\M_2$ and an elementary three-point
symmetrisation. For distinct parameters
$s,t\in\sphere_L$, introduce the
vector-valued kernel
\begin{align*}
    K(s,t)\vcentcolon=
    \frac{\gamma(t)-\gamma(s)}{\abs{\gamma(t)-\gamma(s)}^2}.
\end{align*}
The elementary three-point identity
\begin{align}\label{symmetrisationIntro}
    \frac{1}{2}c\bigl(\gamma(s),\gamma(t),\gamma(u)\bigr)^2
    = K(s,t)\cdot K(s,u) + K(t,s)\cdot K(t,u) + K(u,s)\cdot K(u,t)
\end{align}
is the real-vector analogue of the
symmetrisation underlying Melnikov's
curvature method. It is valid in every
Euclidean dimension.

The individual terms on the right-hand
side are singular along the parameter
diagonals. We therefore introduce the
symmetrically truncated field
\begin{align*}
    F_\epsilon(s) \vcentcolon= \int_{\{t:\epsilon < d_L(s, t)\}} K(s, t)\,\dd t,\qquad\text{and}\qquad
    N_\epsilon \vcentcolon= \int_{\sphere_L}\abs{F_\epsilon(s)}^2\,\dd s,
\end{align*}
where $d_L$ denotes the intrinsic distance on the parameter circle
\begin{align*}
	d_{L}(s, t)\vcentcolon=\min_{k\in\Z}\abs{s-t+kL}.
\end{align*}

The central step is to compare the
truncation defining $N_\epsilon$ with the
fully symmetric truncation of $\M_2$.
The symmetrised part contributes one sixth
of the truncated curvature integral.
The remaining region is a strip along
the third parameter diagonal.
A local expansion of the kernel, using only
the uniform continuity of the tangent,
shows that this strip has the universal
limit $\frac{\pi^2L}{3}$. Consequently,
\begin{align*}
    \lim_{\epsilon\downarrow 0}N_\epsilon
    = \frac{1}{6}\M_2(\gamma)+\frac{\pi^2L}{3},
\end{align*}
with the value $+\infty$ allowed.
In the planar case, the corresponding identity
for the Cauchy transform, including the same
universal diagonal correction, was established
by Tolsa and Verdera
\cite[Theorem~2]{TolsaVerdera2006}.

On the other hand, antisymmetry of $K$
gives the exact pairing identity
\begin{align*}
    \int_{\sphere_L}
    F_\epsilon(s)\cdot\gamma(s)\,\dd s
    =-\frac{L(L-2\epsilon)}2.
\end{align*}
After translating the curve to have zero
mean, the Cauchy--Schwarz inequality and
the periodic Wirtinger inequality imply
\begin{align*}
	\pi^2\frac{(L-2\epsilon)^2}{L} \leq N_\epsilon.
\end{align*}
Combining the two estimates yields
\begin{align*}
	4\pi^2L \leq \M_2(\gamma).
\end{align*}
The equality case forces equality in
Wirtinger's inequality and, together with
arc length parametrisation, characterises
the round circle. Finally, monotonicity
of normalised $L^p$ norms extends the
result to every $p>2$.

\section{Menger curvature and symmetrisation}\label{sectionSymmetrisation}

We begin with the elementary algebraic identity behind the argument. 

\begin{lemma}[Melnikov's symmetrization identity]\label{symmetrisationLemma}
	For any three pairwise distinct points $x, y, z\in\R^{d}$ the following identity holds
	\begin{align}\label{symmetrisationIdentity}
		\frac{1}{2}c(x,y,z)^{2}
		={}&\frac{(y-x)\cdot(z-x)}{\abs{y-x}^{2}\abs{z-x}^{2}}
		+\frac{(x-y)\cdot(z-y)}{\abs{x-y}^{2}\abs{z-y}^{2}}
		+\frac{(x-z)\cdot(y-z)}{\abs{x-z}^{2}\abs{y-z}^{2}}.
	\end{align}
\end{lemma}
\begin{proof}
	Let
	\begin{align*}
		a=\abs{y-z},\qquad b=\abs{z-x},\qquad c=\abs{x-y}
	\end{align*}
	and let $A$ denote the area of the Euclidean triangle spanned by $x, y, z$. By the cosine law,
	\begin{align*}
		(y-x)\cdot(z-x)=\frac{b^{2}+c^{2}-a^{2}}{2},
	\end{align*}
	and analogous identities hold at the other two vertices. The right-hand side of \eqref{symmetrisationIdentity} therefore equals
	\begin{align*}
		\frac{2a^{2}b^{2}+2b^{2}c^{2}+2c^{2}a^{2}-a^{4}-b^{4}-c^{4}}{2a^{2}b^{2}c^{2}}.
	\end{align*}
	Heron's identity gives
	\begin{align*}
		16A^{2}=2a^{2}b^{2}+2b^{2}c^{2}+2c^{2}a^{2}-a^{4}-b^{4}-c^{4}.
	\end{align*}
	For three pairwise distinct points, we know from elementary geometry that
	\begin{align*}
		c(x,y,z)=\frac{4A}{abc}.
	\end{align*}
	This formula also covers the collinear case $A=0$ under the convention $c(x,y,z)=0$. Hence 	
	\begin{align*}
		\frac{1}{2}c(x,y,z)^{2} = \frac{16A^2}{2a^2b^2c^2} = \frac{2a^{2}b^{2}+2b^{2}c^{2}+2c^{2}a^{2}-a^{4}-b^{4}-c^{4}}{2a^{2}b^{2}c^{2}}.
	\end{align*}
\end{proof}

\section{Local form of the kernel}\label{sectionPV}

We establish a local expansion of the kernel $K$ near the parameter diagonal, which will be used to analyse the diagonal correction.
Here, we only need continuity of the tangent $T$. 
It is convenient to formulate the local expansion in terms of its modulus of continuity
\begin{align}\label{definitionOmega}
	\omega(r)
	\vcentcolon=
	\sup\{\abs{T(s)-T(t)}:d_{L}(s,t)\leq r\},
	\qquad 0\leq r\leq \frac{L}{2}.
\end{align}
Since $T$ is continuous on the compact parameter circle, $\omega$ is nondecreasing and $\omega(r)\to0$ as $r\downarrow0$.

\begin{lemma}[Local expansion of the kernel]\label{kernelExpansionLemma}
	Let $\gamma\in C^{1}(\sphere_{L},\R^{d})$ be parametrised by arc length and set $T=\gamma'$. There are $\delta_{0}\in(0,\frac{L}{8})$ and a universal constant $C>0$ such that for every $s\in\sphere_{L}$ and every $0<\abs{h}<\delta_{0}$,
	\begin{align}\label{kernelExpansion}
		K(s,s+h)=\frac{T(s)}{h}+E_{s}(h),
		\qquad
		\abs{E_{s}(h)}\leq C\frac{\omega(\abs{h})}{\abs{h}}.
	\end{align}
	Here and below additions in the parameter are understood modulo $L$.
\end{lemma}
\begin{proof}
	By the fundamental theorem of calculus,
	\begin{align*}
		\gamma(s+h)-\gamma(s)=hT(s)+r_{s}(h),
	\end{align*}
	where $r_s$ is the remainder at $s$. 
	For $s\in\sphere_L$ and sufficiently small $h$, set
	\begin{align*}
		r_s(h) \vcentcolon= \gamma(s+h)-\gamma(s)-hT(s)
		= \int_0^h \bigl(T(s+\tau)-T(s)\bigr)\dd \tau.
	\end{align*}
	By \eqref{definitionOmega} we may estimate
	\begin{align}\label{remainderEstimate}
		\abs{r_{s}(h)} \leq \abs{h}\,\omega(\abs{h}).
	\end{align}
	Choose $\delta_{0}<\frac{L}{8}$ so small that $\omega(\delta_{0})\leq \frac{1}{2}$. 
	Writing $\Delta = \gamma(s+h)-\gamma(s) = hT(s)+r_{s}(h)$, we obtain
	\begin{align}\label{localChordBound}
		\frac{\abs{h}}{2}\leq \abs{hT(s)} -\abs{r_s(h)}\leq \abs{\Delta}
	\end{align}
	for $0<\abs{h}<\delta_{0}$. 
	Consequently,
	\begin{align*}
		K(s, s+h)-\frac{T(s)}h
		= \frac{r_{s}(h)}{\abs{\Delta}^{2}} + T(s)\left(\frac{h}{\abs{\Delta}^{2}}-\frac{1}{h}\right).
	\end{align*}
	By \eqref{remainderEstimate} the first term is bounded by $\frac{4\omega(\abs{h})}{\abs{h}}$. Moreover,
	\begin{align*}
		\abs{h^{2}-\abs{\Delta}^{2}}
		\leq 2\abs{h}\abs{r_{s}(h)}+\abs{r_{s}(h)}^{2}
		\leq \frac{5}{2}\abs{h}^{2}\omega(\abs{h}),
	\end{align*}
	where the last inequality follows from $\omega(\abs{h})\leq\omega(\delta_0)\leq\frac{1}{2}$.
	Together with \eqref{localChordBound} this gives
	\begin{align*}
		\left\lvert\frac{h}{\abs{\Delta}^{2}}-\frac1h\right\rvert
		\leq10\frac{\omega(\abs{h})}{\abs{h}}.
	\end{align*}
	Thus \eqref{kernelExpansion} holds, for example with $C=14$.
\end{proof}

\section{Symmetric truncation and the diagonal correction}\label{sectionTruncations}

We introduce a truncated kernel integral and
decompose its squared $L^2$-norm into a
symmetrised curvature term and a diagonal correction.

For this section we require $\gamma$ to be embedded.
Let
\begin{align*}
	d_{L}(s,t)\vcentcolon=\min_{k\in\Z}\abs{s-t+kL}
\end{align*}
be the intrinsic distance on $\sphere_{L}$. For $0< \epsilon < \frac{L}{2}$ define the truncated singular integral
\begin{align*}
	F_{\epsilon}(s)	\vcentcolon= \int_{\epsilon < d_{L}(s, t)}K(s, t)\,\dd t
\end{align*}
and set
\begin{align*}
	N_{\epsilon} \vcentcolon= \int_{\sphere_{L}}\abs{F_{\epsilon}(s)}^{2}\,\dd s.
\end{align*}
Since $\gamma$ is embedded and continuous,
the kernel $K$ is bounded on $\{(s,t):\epsilon\leq d_L(s, t)\}$.
Consequently, $F_\epsilon$ and $N_\epsilon$ are finite for every $0 < \epsilon < \frac{L}{2}$.
Expanding the square and using the bilinearity of the scalar product, we obtain
\begin{align*}
	N_\epsilon
	=
	\int_{\sphere_L}
		\int_{\epsilon < d_L(s, t)}
			\int_{\epsilon < d_L(s, u)}
				K(s, t)\cdot K(s, u)
			\,\dd u
		\,\dd t
	\,\dd s.
\end{align*}
The fully symmetric truncation is
\begin{align*}
	D_{\epsilon} \vcentcolon= \left\{(s, t, u)\in\sphere_{L}^{3}:\epsilon< d_{L}(s, t),\ \epsilon< d_{L}(s, u),\ \epsilon<d_{L}(t, u)\right\}.
\end{align*}
We split
\begin{align*}
	N_{\epsilon} = A_{\epsilon} + B_{\epsilon},
\end{align*}
where
\begin{align*}
	A_{\epsilon} \vcentcolon= \iiint_{D_{\epsilon}}K(s, t)\cdot K(s, u)\,\dd s\,\dd t\,\dd u,
\end{align*}
and
\begin{align}\label{BepsilonDefinition}
	B_{\epsilon} \vcentcolon= \iiint_{\substack{\epsilon < d_{L}(s, t),\ \epsilon< d_{L}(s, u)\\d_{L}(t, u)\leq\epsilon}}
	K(s, t)\cdot K(s, u)\,\dd s\,\dd t\,\dd u.
\end{align}
We will show
\begin{align*}
	A_\epsilon\xrightarrow[\epsilon\downarrow 0]{}\frac{1}{6}\M_{2}(\gamma)\qquad\text{and}\qquad
	B_\epsilon\xrightarrow[\epsilon\downarrow 0]{}\frac{\pi^{2}}3L,
\end{align*}
so that
\begin{align}\label{eq:limieNepsilon}
	N_\epsilon\xrightarrow[\epsilon\downarrow 0]{}\frac{1}{6}\M_{2}(\gamma) + \frac{\pi^{2}}3L.
\end{align}

\begin{lemma}[The symmetric part]\label{symmetricPartLemma}
	Let $\gamma:\sphere_L\to\R^d$ be an embedded, arc length parametrised continuous curve.
	Then, for $0<\epsilon<\frac{L}{2}$
	\begin{align}\label{AepsilonMenger}
		A_{\epsilon} = \frac{1}{6}\iiint_{D_{\epsilon}}c\bigl(\gamma(s),\gamma(t),\gamma(u)\bigr)^{2}\,\dd s\,\dd t\,\dd u.
	\end{align}
	Consequently,
	\begin{align}\label{AepsilonLimit}
		\lim_{\epsilon\downarrow0}A_{\epsilon} = \frac{1}{6}\M_{2}(\gamma),
	\end{align}
	where the value $+\infty$ is allowed a priori.
\end{lemma}
\begin{proof}
	The set $D_{\epsilon}$ is invariant under every permutation of $s, t, u$. Therefore the integrals over $D_{\epsilon}$ of the three summands on the right-hand side of \eqref{symmetrisationIntro} coincide and all equal $A_{\epsilon}$. Integrating \eqref{symmetrisationIntro} gives \eqref{AepsilonMenger}. As $\epsilon\downarrow0$, the sets $D_{\epsilon}$ increase to $\sphere_{L}^{3}$ up to the three parameter diagonals, which have measure zero. Since $c^{2}\geq0$, monotone convergence yields \eqref{AepsilonLimit}.
\end{proof}

A closely related computation appears in \cite[Lemma~8]{TolsaVerdera2006}, where Tolsa and Verdera obtain the same universal constant $\frac{\pi^2}{3}$ for Euclidean truncations of the planar Cauchy kernel.

\begin{lemma}[Universal diagonal contribution]\label{diagonalLemma}
	Let $\gamma\in C^{1}(\sphere_{L}, \R^{d})$ be parametrised by arc length and embedded. Then
	\begin{align}\label{BepsilonLimit}
		\lim_{\epsilon\downarrow0}B_{\epsilon} = \frac{\pi^{2}}3L.
	\end{align}
\end{lemma}
\begin{proof}
	Fix $0<\delta<\delta_{0}$ so small that $2\delta < \frac{L}{2}$, and assume $0<\epsilon<\frac{\delta}{2}$. Decompose
	\begin{align*}
		B_{\epsilon}= B_{\epsilon,\delta}^{\mathrm{loc}} + B_{\epsilon,\delta}^{\mathrm{far}},
	\end{align*}
	where the local term is restricted by
	\begin{align*}
		d_{L}(s, t) < \delta,\qquad d_{L}(s, u)<\delta
	\end{align*}
	and $B_{\epsilon,\delta}^{\mathrm{far}}$ contains the complement within the domain of \eqref{BepsilonDefinition}.
	Throughout the proof, $C>0$ denotes a generic constant that may change from one occurrence to the next but is independent of $\epsilon$, $\delta$, and the integration variables.\\

	\emph{Step 1: The nonlocal contribution vanishes.}
		If a triple belongs to the far part, then
		\begin{align*}
			\delta \leq \max\{d_{L}(s, t), d_{L}(s, u)\}.
		\end{align*}
		Since $d_{L}(t, u) \leq \epsilon < \frac{\delta}{2}$, the triangle inequality implies
		\begin{align*}
			\frac{\delta}{2}\leq d_{L}(s, t),\qquad
			\frac{\delta}{2}\leq d_{L}(s, u).
		\end{align*}
		As $\gamma$ is continuous and injective, compactness gives
		\begin{align*}
			0 < m_{\delta} \vcentcolon= \min_{\{(s, t) : \frac{\delta}{2} \leq d_{L}(s, t)\}} \abs{\gamma(s)-\gamma(t)}.
		\end{align*}
		Hence
		\begin{align*}
			\abs{K(s,t)\cdot K(s,u)}\leq m_{\delta}^{-2}
		\end{align*}
		on the far part. For fixed $t$, the set $\{u:d_{L}(t, u)\leq\epsilon\}$ has measure $2\epsilon$. Consequently,
		\begin{align}\label{farPartEstimate}
			\abs{B_{\epsilon,\delta}^{\mathrm{far}}}
			\leq\frac{2L^{2}}{m_{\delta}^{2}}\epsilon
			\xrightarrow[\epsilon\downarrow 0]{}0.
		\end{align}

	\emph{Step 2: Local coordinates.}
		For the local part there are unique $h,k\in(-\delta,\delta)$ such that
		\begin{align*}
			t=s+h,\qquad u=s+k\qquad(\bmod L).
		\end{align*}
		Since $\abs{h-k}<2\delta<\frac{L}{2}$, one has $d_{L}(t,u)=\abs{h-k}$. Thus 
		\begin{align}\label{localBRepresentation}
			B_{\epsilon,\delta}^{\mathrm{loc}}
			= \int_{0}^{L} \iint_{\Omega_{\epsilon,\delta}} K(s, s+h)\cdot K(s, s+k)\,\dd h\,\dd k\,\dd s,
		\end{align}
		where
		\begin{align*}
			\Omega_{\epsilon,\delta} \vcentcolon= \{(h,k):\epsilon<\abs{h}<\delta,\ \epsilon<\abs{k}<\delta,\ \abs{h-k}\leq\epsilon\}.
		\end{align*}
		By Lemma \ref{kernelExpansionLemma}, 
		\begin{align}\label{localProductExpansion}
			K(s, s+h)\cdot K(s, s+k)
			= \frac1{hk}+R_{s}(h, k).
		\end{align}
		Since $\gamma$ is parametrised by arc length, we have $\omega\leq 2$ and
		\begin{align}\label{RemainderProductEstimate}
			\abs{R_{s}(h,k)}
			\leq
			C\frac{\omega(\abs{k})}{\abs{h}\abs{k}}
			+C\frac{\omega(\abs{h})}{\abs{h}\abs{k}}.
		\end{align}

	\emph{Step 3: The error in \eqref{localProductExpansion} vanishes.}
		On $\Omega_{\epsilon,\delta}$ the variables $h$ and $k$ necessarily have the same sign. Indeed, opposite signs would imply $\abs{h-k}=\abs{h}+\abs{k}>2\epsilon$. It therefore suffices to estimate the positive quadrant and multiply by two. We treat the first mixed term in \eqref{RemainderProductEstimate}; the second one is symmetric. On the range $\epsilon<h<2\epsilon$ one has $\epsilon<k<3\epsilon$, hence
		\begin{align*}
			\iint_{\substack{\epsilon<h<2\epsilon,\ \epsilon<k<\delta\\\abs{h-k}\leq\epsilon}}
			\frac{\omega(k)}{hk}\,\dd h\,\dd k
			\leq C\,\omega(3\epsilon)
			\xrightarrow[\epsilon\downarrow0]{}0.
		\end{align*}
		If $2\epsilon\leq h<\delta$, then $\abs{h-k}\leq\epsilon$ implies $\frac{h}{2}\leq k\leq \frac{3h}{2}$, and the admissible $k$-interval has length at most $2\epsilon$. Consequently,
		\begin{align*}
			\iint_{\substack{2\epsilon\leq h<\delta,\ \epsilon<k<\delta\\\abs{h-k}\leq\epsilon}}
			\frac{\omega(k)}{hk}\,\dd h\,\dd k
			\leq C\epsilon\int_{2\epsilon}^{\delta}\frac{\omega(\frac{3h}{2})}{h^{2}}\,\dd h.
		\end{align*}		
		For $2\epsilon \leq h < \delta$, the substitution $h = \epsilon x$ yields
		\begin{align*}
			\epsilon\int_{2\epsilon}^{\delta} \frac{\omega(\frac{3h}{2})}{h^2}\,\dd h
			= \int_2^{\frac{\delta}{\epsilon}} \frac{\omega(\frac{3\epsilon x}{2})}{x^2}\,\dd x
			\xrightarrow[\epsilon\downarrow 0]{} 0
		\end{align*}
		by dominated convergence, since $\omega(r)\to 0$ as $r\downarrow 0$.
		The preceding estimates are uniform in $s$. Consequently,
		\begin{align}\label{errorVanishes}
			\int_{0}^{L}\iint_{\Omega_{\epsilon,\delta}}
			R_{s}(h,k)\,\dd h\,\dd k\,\dd s
			\xrightarrow[\epsilon\downarrow0]{}0.
		\end{align}

	\emph{Step 4: Computation of the universal main term.}
		It remains to evaluate
		\begin{align*}
			J_{\epsilon,\delta}
			\vcentcolon=
			\iint_{\Omega_{\epsilon,\delta}}\frac{\dd h\,\dd k}{hk}.
		\end{align*}
		The positive and negative quadrants give identical contributions. On the positive quadrant, the scaling $h=\epsilon x$, $k=\epsilon y$ yields
		\begin{align*}
			J_{\epsilon,\delta}^{+}
			=
			\iint_{\substack{1<x<\frac{\delta}{\epsilon}\\1<y<\frac{\delta}{\epsilon}\\\abs{x-y}\leq1}}
			\frac{\dd x\,\dd y}{xy}.
		\end{align*}
		By monotone convergence,
		\begin{align*}
			\MoveEqLeft
			\lim_{\epsilon\downarrow0}J_{\epsilon,\delta}^{+}
			= \iint_{\substack{1< x\\1 < y\\\abs{x-y}\leq1}}
			\frac{\dd x\,\dd y}{xy}
			= 2\iint_{\substack{1< x < y\\\abs{x-y}\leq1}} \frac{\dd x\,\dd y}{xy}\\
			& = 2\iint_{1< x < y \leq x+1 } \frac{\dd x\,\dd y}{xy}
			= 2\int_{1}^{\infty}\frac{1}{x}\int_{x}^{x+1}\frac{\dd y}{y}\,\dd x\\
			& = 2 \int_{1}^{\infty}\frac{\log(1+\frac{1}{x})}{x}\,\dd x
			= 2\int_{0}^{1}\frac{\log(1+v)}{v}\,\dd v\\
			& = 2\sum_{m=1}^\infty\frac{(-1)^{m+1}}{m}\int_{0}^{1}v^{m-1}\,\dd v
			= 2\sum_{m=1}^{\infty}\frac{(-1)^{m+1}}{m^{2}}
			=\frac{\pi^{2}}6.
		\end{align*}
		Here we used the Taylor series of the logarithm, termwise integration justified by absolute	integrability, and Euler's solution to the Basel problem.
		Thus
		\begin{align}\label{Jlimit}
			\lim_{\epsilon\downarrow0}J_{\epsilon,\delta}=\frac{\pi^{2}}3.
		\end{align}
		Combining \eqref{localBRepresentation}, \eqref{localProductExpansion}, \eqref{errorVanishes}, and \eqref{Jlimit} gives
		\begin{align*}
			B_{\epsilon,\delta}^{\mathrm{loc}}\xrightarrow[\epsilon\downarrow 0]{}\frac{\pi^{2}}3L.
		\end{align*}
		Together with \eqref{farPartEstimate}, this proves \eqref{BepsilonLimit}.
\end{proof}

\section{The sharp quadratic inequality}\label{sectionQuadratic}

We now combine a truncated pairing identity with the periodic Wirtinger inequality to establish the quadratic lower bound 
and characterise the equality case.

Since the kernel $K$, the truncated field $F_\epsilon$, and the integral Menger curvature are invariant under translations of $\gamma$, we may replace $\gamma$ by $\gamma-\bar\gamma$, where
\begin{align*}
	\bar\gamma \vcentcolon= \frac{1}{L}\int_0^L\gamma(s)\,\dd s.
\end{align*}
Thus, without loss of generality, we may assume
\begin{align}\label{zeroMean}
	\int_0^L\gamma(s)\,\dd s=0.
\end{align}

\begin{lemma}[Truncated pairing identity]\label{pairingLemma}
	Let $\gamma:\sphere_L\to\R^d$ be an embedded, arc length parametrised continuous curve.
	For every $0 < \epsilon < \frac{L}{2}$,
	\begin{align*}
		\int_{0}^{L}F_{\epsilon}(s)\cdot\gamma(s)\,\dd s
		= -\frac{L(L-2\epsilon)}2.
	\end{align*}
\end{lemma}
\begin{proof}
	Antisymmetry $K(t, s)=-K(s, t)$ gives
	\begin{align*}
		\MoveEqLeft
		2\int_{0}^{L}F_{\epsilon}(s)\cdot\gamma(s)\,\dd s
		= \int_0^L\int_{\epsilon < d_{L}(s, t)} K(s, t)\cdot\big(\gamma(s)-\gamma(t)\big)\,\dd s\,\dd t\\
		& =	-\int_0^L\int_{\epsilon < d_{L}(s, t)}1\,\dd s\,\dd t
		= -L(L-2\epsilon),
	\end{align*}
	because
	\begin{align*}
		K(s, t)\cdot\bigl(\gamma(s)-\gamma(t)\bigr)=-1.
	\end{align*}
\end{proof}

\begin{proposition}[Lower bound for the truncated field]
	\label{FEstimateProposition}
	Let $\gamma\in C^{1}(\sphere_{L}, \R^{d})$ be parametrised by arc length and embedded. Then
	\begin{align}\label{FEstimate}
		\pi^{2}\frac{(L-2\epsilon)^{2}}{L}\leq N_{\epsilon}.
	\end{align}
	In particular,
	\begin{align}\label{liminfNEstimate}
		\pi^{2}L \leq \liminf_{\epsilon\downarrow0}N_{\epsilon}.
	\end{align}
\end{proposition}
\begin{proof}
	Translate $\gamma$ so that \eqref{zeroMean} holds.	By Lemma \ref{pairingLemma} and the Cauchy--Schwarz inequality,
	\begin{align}\label{CauchySchwarzPairing}
		\frac{L^{2}(L-2\epsilon)^{2}}4
		= \abs{	\int_{0}^{L} F_{\epsilon}(s)\cdot\gamma(s) \,\dd s }^{2}
		\leq N_{\epsilon} \int_{0}^{L} \abs{\gamma(s)}^{2} \,\dd s.
	\end{align}
	The classical periodic Wirtinger inequality (see \cite[Theorem~258]{HardyLittlewoodPolya1952}, the version for arbitrary period $L$ follows from the classical $2\pi$-periodic inequality by rescaling) states that every $L$-periodic scalar $C^{1}$ function $f$ with zero mean satisfies
	\begin{align*}
		\int_{0}^{L}\abs{f(s)}^{2}\,\dd s
		\leq\frac{L^{2}}{4\pi^{2}} \int_{0}^{L} \abs{f'(s)}^{2} \,\dd s.
	\end{align*}
	Equality holds if and only if
	\begin{align*}
		f(s) = a\cos\left(\frac{2\pi s}{L}\right) +	b\sin\left(\frac{2\pi s}{L}\right)
	\end{align*}
	for some $a, b\in\R$. 
	Applying this inequality to each component of $\gamma$ and summing, we obtain,
	\begin{align}\label{WirtingerGamma}
		\int_{0}^{L} \abs{\gamma(s)}^{2} \,\dd s
		\leq \frac{L^{2}}{4\pi^{2}} \int_{0}^{L} \abs{\gamma'(s)}^{2} \,\dd s
		= \frac{L^{3}}{4\pi^{2}}.
	\end{align}
	Combining \eqref{CauchySchwarzPairing} and \eqref{WirtingerGamma} yields
	\begin{align*}
		\frac{L^{2}(L-2\epsilon)^{2}}4 \leq	N_{\epsilon} \frac{L^{3}}{4\pi^{2}},
	\end{align*}
	and hence
	\begin{align*}
		\pi^{2}\frac{(L-2\epsilon)^{2}}{L} \leq N_{\epsilon}.
	\end{align*}
	This proves \eqref{FEstimate}. Letting $\epsilon\downarrow 0$ gives \eqref{liminfNEstimate}.
\end{proof}

\begin{theorem}[Sharp quadratic Menger-curvature inequality]\label{quadraticTheorem}
	Let $\gamma\in C^{1}(\sphere_{L}, \R^{d})$ be parametrised by arc length and embedded. Then
	\begin{align}\label{quadraticInequality}
		4\pi^{2}L \leq \M_{2}(\gamma).
	\end{align}
	Equality holds if and only if $\gamma$ parametrises a round circle.
\end{theorem}
\begin{proof}
	If $\M_{2}(\gamma)=\infty$, there is nothing to prove.
	As a consequence of Lemmas \ref{symmetricPartLemma} and \ref{diagonalLemma} in section \ref{sectionTruncations} we have shown \eqref{eq:limieNepsilon}, i.e.
	\begin{align*}
		\lim_{\epsilon\downarrow0}N_{\epsilon}
		= \frac{1}{6}\M_{2}(\gamma)+\frac{\pi^{2}}3L.
	\end{align*}
	Combining this with \eqref{liminfNEstimate} yields
	\begin{align*}
		\pi^{2}L\leq \frac{1}{6}\M_{2}(\gamma)+\frac{\pi^{2}}3L,
	\end{align*}
	which is equivalent to \eqref{quadraticInequality}. A round circle attains equality by \eqref{circlevalue}.

	Conversely, suppose equality holds in \eqref{quadraticInequality}. Then
	\begin{align}\label{NequalityLimit}
		N_{\epsilon}\xrightarrow[\epsilon\downarrow 0]{}\pi^{2}L.
	\end{align}
	After translating $\gamma$ to have mean zero, \eqref{CauchySchwarzPairing} and \eqref{NequalityLimit} imply, upon letting $\epsilon\downarrow0$,
	\begin{align*}
		\frac{L^{3}}{4\pi^{2}}
		\leq \int_{0}^{L}\abs{\gamma(s)}^{2}\,\dd s.
	\end{align*}
	Together with \eqref{WirtingerGamma}, equality must hold in the periodic Wirtinger inequality. Hence there exist $a ,b\in\R^{d}$ such that
	\begin{align*}
		\gamma(s)=a\cos\left(\frac{2\pi s}{L}\right) + b\sin\left(\frac{2\pi s}{L}\right).
	\end{align*}
	Evaluating $\abs{\gamma'(s)} = 1$ at $s = 0$, $s = \frac{L}{4}$, and $s = \frac{L}{8}$, respectively, gives
	\begin{align*}
		\abs{b}^{2} = \frac{L^{2}}{4\pi^{2}},\qquad
		\abs{a}^{2} = \frac{L^{2}}{4\pi^{2}}\qquad\text{and}\qquad
		a\cdot b & =0.
	\end{align*}
	Thus $a$ and $b$ are orthogonal vectors of equal length and $\gamma$ parametrises a round circle.
\end{proof}

\section{Higher exponents and proof of the main theorem}\label{sectionHigherExponents}

The passage from $p=2$ to $p>2$ uses only the monotonicity of normalised $L^{p}$ norms.

\begin{proof}[Proof of Theorem \ref{mainTheorem}]
	The case $p=2$ is Theorem \ref{quadraticTheorem}. Let $p>2$. If $\M_{p}(\gamma) = \infty$, the assertion is trivial. Otherwise, since $\sphere_{L}^{3}$ has measure $L^{3}$, monotonicity of normalised $L^{p}$ norms gives
	\begin{align*}
		\left(\frac{\M_{2}(\gamma)}{L^{3}}\right)^{\frac{1}{2}}\leq 
		\left(\frac{\M_{p}(\gamma)}{L^{3}}\right)^{\frac{1}{p}}.
	\end{align*}
	Using \eqref{quadraticInequality}, we obtain
	\begin{align*}
		2\pi L^{\frac{3}{p}-1} = L^{\frac{3}{p}-\frac{3}{2}}\left(4\pi^{2}L\right)^{\frac{1}{2}}
		\leq \M_{p}(\gamma)^{\frac{1}{p}}.
	\end{align*}
	Raising to the $p$-th power yields \eqref{mainInequality}. A round circle attains equality by \eqref{circlevalue}. Conversely, equality in \eqref{mainInequality} forces equality in the chain
	\begin{align*}
		\frac{2\pi}{L} \leq \left(\frac{\M_{2}}{L^{3}}\right)^{\frac{1}{2}}
		\leq \left(\frac{\M_{p}}{L^{3}}\right)^{\frac{1}{p}},
	\end{align*}
	and in particular $\M_{2} = 4\pi^{2}L$. Theorem \ref{quadraticTheorem} then implies that $\gamma$ is a round circle.
\end{proof}

\begin{proof}[Proof of Corollary \ref{basicCorollary}]
	If $\M_p(\gamma) = \infty$, there is nothing to prove. Otherwise, the self-avoidance and regularity results of Strzelecki, Szuma{\'n}ska and von der Mosel \cite[Theorem~1.2]{StrzeleckiSzumanskaVdM2010}, \cite[Theorems~2.2 and~2.3]{StrzeleckiVdM2013} imply that $\gamma$ is embedded and belongs to $C^{1, 1-3/p}$.
	Here, the condition	$\HM^1(\gamma(\sphere_L)) = L$ excludes multiple coverings. Although the cited result is stated
	for curves in $\R^3$, their proofs are	dimension-independent and extend to $\R^d$ for $d \geq 2$.
	The inequality now follows directly from Theorem \ref{mainTheorem}.
	
	That $\M_p$ is unknot-detecting follows from the known minimisability
	and strongness of $\M_p$ for $p>3$;	see \cite[Theorem~3.2]{StrzeleckiVdM2013}.
	Indeed, only finitely many knot types occur
	below any fixed energy level, and the
	minimum in each knot class is attained.
	Since the round circle is the unique global
	minimiser, the infimum over nontrivial knot
	classes is strictly larger than its energy.
\end{proof}

\subsection*{Acknowledgement}
The authors gratefully acknowledge the extensive and
substantive mathematical assistance provided by OpenAI's
ChatGPT, which played a significant role in the development
and refinement of the central arguments of this paper.
In particular, ChatGPT contributed to exploring proof
strategies, developing and critically examining mathematical
arguments, working through technical details, and improving
the structure and exposition of the manuscript.
The authors take full responsibility for the mathematical
correctness and final content of the paper.

\setlength{\parindent}{0cm}

\bibliography{menger_circle}{}
\bibliographystyle{abbrv}

\end{document}